\documentclass[12pt]{article}

\usepackage[lmargin = 2cm, rmargin = 2cm, bmargin = 2cm,tmargin=2cm]{geometry}

\usepackage{titlesec}

\usepackage{microtype}
  
\usepackage{enumitem}

\usepackage{xfrac}

\usepackage{subfig}

\usepackage{titlesec}

\usepackage{graphicx}

\usepackage{amssymb}

\usepackage{amsmath}

\usepackage{cancel}

\usepackage{amsthm}

\usepackage{accents}

\usepackage{wrapfig}

\usepackage{mathrsfs}

\usepackage{eufrak}

\usepackage{tikz-cd}

\usepackage{xcolor}

\usepackage{hyphenat}

\usepackage{fancyhdr}

\usepackage{url}

\usepackage{mathdots}

\usepackage{etoolbox}

\usepackage{tgtermes}

\usetikzlibrary{decorations.pathmorphing}

\definecolor{ggreen}{rgb}{0,0.75,0.08}

\theoremstyle{plain}

\theoremstyle{definition}

\newtheorem{theorem}{Theorem}[section]

\newtheorem{lemma}[theorem]{Lemma}

\newtheorem{definition}[theorem]{Definition}

\newcounter{dummy5} 
\newtheorem{thm5}[dummy5]{Theorem}

\newcounter{qn5counter} 
\newtheorem{qn5}[qn5counter]{Question}

\newcounter{claim5counter} 
\newtheorem{claim5}[claim5counter]{Claim}

\newtheorem{question*}{Question}

\newcommand{\A}{\ensuremath{\alpha}}

\newcommand{\K}{\ensuremath{\kappa}}
\newcommand{\B}{\ensuremath{\beta}}

\newcommand{\E}{\ensuremath{\varepsilon}}
\newcommand{\W}{\ensuremath{\omega}}
\newcommand{\WW}{\ensuremath{\Omega}}

\newcommand{\G}{\ensuremath{\gamma}}
\newcommand{\D}{\ensuremath{\delta}}
\newcommand{\GG}{\ensuremath{\Gamma}}

\newcommand{\NN}{\ensuremath{\mathbb N}}

\newcommand{\0}{\ensuremath{\varnothing}}

\newcommand{\cL}{\ensuremath{\mathcal L}}

\newcommand{\cD}{\ensuremath{\mathcal D}}

\newcommand{\cK}{\ensuremath{\mathcal K}}
\newcommand{\cB}{\ensuremath{\mathcal B}}

\newcommand{\cP}{\ensuremath{\mathcal P}}

\newcommand{\DD}{\ensuremath{\partial}}
\newcommand{\All}{\ensuremath{\forall}}

\newcommand{\cn}{\ensuremath{\frak c}}

\newcommand{\wt}{\ensuremath{\widetilde}}

\begin{document}

\openup 0.6em

\fontsize{13}{5}
\selectfont

	\begin{center}\LARGE Separable Indecomposable Continuum with\\ Exactly One Composant
	\end{center}
	
	\begin{align*}
	\text{\Large Daron Anderson }  \qquad \text{\Large Trinity College Dublin. Ireland }  
	\end{align*} 
	\begin{align*} \text{\Large andersd3@tcd.ie} \qquad \text{\Large Preprint September 2018}  
	\end{align*}$ $\\

	\begin{center}
		\textbf{ \large Abstract}
	\end{center}
	
	\noindent
	Indecomposable continua with one composant are \textit{large} in the sense of being non-metrisable.
	We adapt the method of Smith \cite{Smith1} to construct an example which is \textit{small} in the sense of being separable.

	\section{Introduction}
	
	\noindent By a \textit{Bellamy continuum} we mean an indecomposable Hausdorff continuum with exactly one composant.
	To date there are only two known classes of Bellamy continua. Examples of the first class are Stone-\v Cech remainders of certain locally-compact spaces. The remainder is a continuum and under certain set-theoretic assumptions \cite{waves,NCF2,BitLike} it has exactly one composant.	
	
	Examples of the second class are those obtained by Bellamy and Smith 
	by carefully constructing an $\W_1$-chain of metric continua and retractions,
	so the inverse limit has exactly two composants. They then select one point from each composant, and identify the two points to get exactly one composant \cite{one, smith2,Smith1, HI2Composants}.
	
	It is well-known that each Bellamy continuum is \textit{large} in the sense of being non-metrisable \cite{nadlerbook}.
	This raises the question of whether a Bellamy continuum can be \textit{small} in the sense of being separable.
	The first class mentioned above provides no examples.
	Indeed Corollary 5.6 of \cite{CSbook} says the Stone-\v Cech remainder of any well-behaved locally-compact Hausdorff space 
	is $\aleph_1$-cellular hence non-separable.
	
	In this paper we modify the inverse-system of Smith \cite{Smith1} to produce a separable Bellamy continuum in the second class.
	
	Section 3 contains preliminary results about separability of certain inverse limits of metric continua.
	Section 4 applies the results to get an inverse system of metric continua and retractions.
	The modification itself is minor and only needed to make the limit separable $-$ it is inessential to showing the limit has exactly two composants.
	We obtain a separable Bellamy continuum by \textit{spot-welding} as usual.

	Section 5 shows both composants are non-metrisable.
	One composant is separable and the other is non-separable.
	Hence our example is separable but not hereditarily separable.
	The problem is open whether there exists an hereditarily separable Bellamy continuum.

	Section 6 extends familiar results about Bellamy continua by showing 
	each hereditarily unicoherent metric continuum is a retract of a separable Bellamy continuum.
	The problem is open whether the same holds for each separable continuum.
	Section 7 shows our modification is necessary, as the original limit of Smith is non separable.

	\section{Terminology and Notation}
	\noindent
	Throughout $X$ is a continuum. 
	That means a nondegenerate compact and connected Hausdorff space.
	For background on metric continua see \cite{kur2} and \cite{nadlerbook}.
	The results cited here have analagous proofs for non-metric continua.

	We call $X$ \textit{separable} to mean it has a dense countable subset,
	\textit{hereditarily separable} to mean every subset is separable,
	and \textit{dense-hereditarily separable} to mean every dense subset is separable.
	Metric continua are separable but the converse fails in general.
	For each cardinal $\A$  we say $X$ is $\A$-\textit{cellular} to mean it admits a family of $\A$-many pairwise disjoint open subsets.
	Clearly each $\aleph_1$-cellular continuum is non-separable.

	For a subset $S \subset X$ denote by $S^\circ$ and $\overline S$ the interior and closure of $S$ respectively. 
	By \textit{boundary bumping} we mean the principle that, for each closed $E \subset X$, each component $C$ of $E$ meets $\DD E = \overline E \cap \overline {X-E}$.
	For the non-metric proof see $\S$47, III Theorem 2 of \cite{kur2}. 
	One corollary of boundary bumping is that the point $p \in X$ is in the closure of each continuum component of $X-p$.

	Throughout all maps between continua are assumed to be continuous.
	We call $f:Y \to X$ a \textit{retraction} to mean $X$ is a subspace of the continuum $Y$ and the restriction of $f$ to $X$ is the identity.
	The partition $\cP$ of $X$ into closed subsets is called \textit{upper semicontinuous} to mean the following:
	For each $P \in \cP$ and open $U \subset X$ containing $P$ there is open $V \subset U$ with $P \subset V$ and $V$ a union of elements of $\cP$.
	Upper semicontinuity of the partition is equivalent to the quotient space $X/\cP$ being a continuum.

	For $b \in X$ we omit the curly braces and write $X-b$ instead of $X-\{b\}$ without confusion.
	For $a,b \in X$ we say $X$ is \textit{irreducible} about $\{a,b\}$ to mean no proper subcontinuum of $X$ contains $\{a,b\}$.
	For $a,b \in X$ we write $[a,b]$ for the intersection of all subcontinua that contain $\{a,b\}$.
	Note $[a,b]$ is not in general connected as the interval notation suggests.
	Clearly $h\big ([a,b] \big ) = [h(a),h(b)]$ for each $a,b \in X$ and homeomorphism $h:X \to Y$. 
	
	We say $X$ is \textit{indecomposable} to mean it is not the union of two proper subcontinua.
	Equivalently each proper subcontinuum is nowhere dense.
	The \textit{composant} $\K(x)$ of the point $x \in X$ is the union of all proper subcontinua that have $x$ as an element.
	Indecomposable metric continua are partitioned into $\cn$ many pairwise disjoint composants \cite{Ccomposants}.
	In case $\K(x) \ne \K(y)$ then $X$ is irreducible about $\{x,y\}$.
	There exist indecomposable non-metric continua with exactly one composant \cite{one} henceforth called \textit{Bellamy continua}.
	
	We say $X$ is \textit{hereditarily unicoherent} to mean 
	the intersection of any two subcontinua of $X$ is empty or connected.
	Equivalently each interval $[a,b]$ is a continuum.
	It then follows $[a,b]$ is the unique subcontinuum of $X$ irreducible about $\{a,b\}$.

	\begin{figure}[!h]
		\centering
		\begin{tikzpicture}

		\usetikzlibrary{snakes}

		\begin{scope}[xscale = -40, yscale = 2]
		
		\draw[smooth,domain=0.06:0.285,samples=1000] plot (\x, { (2*sin((5/\x)r)  +2) });

		\draw[ultra thick, magenta] (0.06,0) --(0.06,4);
		
		\node[circle, fill = orange, minimum width = 0.15cm, inner sep = 0cm] at (0.285,0.075) {};
		\node[below]  at (0.285,-0.1) {$a$}; 
		
		\node[left] at (0.285,2) {$G$};

		\node[above] at (0.06, 4.25) {$I$};

		\end{scope}
		
		\end{tikzpicture}
		\caption{The $\sin(1/x)$ continuum}\label{SinContinuum}
	\end{figure}
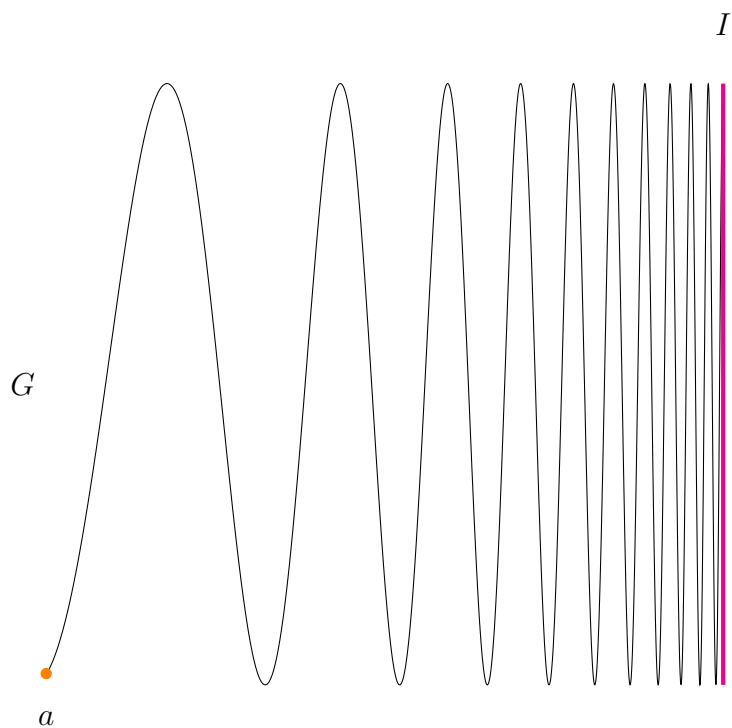

	Throughout the $\sin(1/x)$ continuum is the metric continuum defined as the union of the graph 
	$G =  \big \{ \big (x,\sin(1/x) \big ): -\pi/2 \le x < 0 \big \}$ and the arc $I = \{0\} \times [-1,1]$.
	The composant $\K(a)$ of the endpoint $a = (-\pi/2,-1)$ is equal to $G$ and its every nondegenerate subcontinuum is equal to the closure of its interior.

	The proper subcontinuum $R \subset X$ is called a \textit{rung} to mean each other subcontinuum $K \subset X$ 
	is either disjoint from, contained in, or contains $R$.
	For example the limiting arc $I$ is a rung of the $\sin(1/x)$ continuum.
	By a \textit{ladder} on $X$ we mean a nested collection of rungs of $X$ with dense union.
	Continua that admit ladders are rare.
	For example the $\sin(1/x)$ continuum admits no ladders.
	Note what we call rungs are sometimes called \textit{terminal subcontinua}, 
	but that term also has several unrelated meanings across continuum theory \cite{TerminalContinua}.

	Throughout 
	$\W = \{0,1,2, \ldots\}$ is the first infinite ordinal and $\W_1$ the first uncountable ordinal.
	Each proper initial segment of $\W_1$ is countable and each countable subset has an upper bound.
	Each countable ordinal has a cofinal subset order-isomorphic to $\W$.
	For the ordered set $\WW$ we say $\Psi \subset \WW$ is \textit{cofinal} to mean it has no upper bound in $\WW$.
	
	The poset $\WW$ is called \textit{directed} to mean for each $\G,\B \in \WW$ there is $\A \in \WW$ with $\G,\B \le \A$.
	Note most authors require $\G,\B < \A$. The stronger condition prohibits $\WW$ having a top element.
	In this paper it is convenient to allow a directed set to have a top element.
	
	An \textit{inverse system} over the directed set $\WW$ consists of the following data: $(1)$ a family of topological spaces $T(\A)$ for each $\A \in \WW$ and $(2)$ a family of continuous maps $f^\A_\B : T(\A) \to T(\B)$ for each $\B \le \A$ such that $(3)$ we have $f^\B_\G \circ f^\A_\B = f^\A_\G$ whenever $\G\le \B \le \A$. The property $(3)$ is called  \textit{commutativity of the diagram}.
	The \textit{inverse limit} $T$ of the system is the space 
	
	\begin{center}                                             
		$\displaystyle \varprojlim \{T(\A); f^\A_\B: \A,\B \in \WW\} = \Big \{(x_\A) \in \prod_{\A \in \WW} T(\A) : f^\A_\B(x_\A) = x_\B$ for all $ \B \le \A \Big \}$.
	\end{center}
	
	The functions $f^\A_\B$ are called the \textit{bonding maps}.
	Write $\pi_\B : T \to T(\B)$ for the restriction of the projection $\prod_\A T(\A) \to T(\B)$.
	If each bonding map is surjective so is each $\pi_\B$ and we call the inverse system (limit) \textit{surjective}.
	In case $\WW$ has top element $\infty$ the inverse limit is a copy of $T(\infty)$.
	The inverse limit of a system of continua is a continuum.

	\section{The Successor Stage}\label{5.3}

	We use transfinite recursion to construct the eponymous indecomposable continuum as the limit of a system \mbox{$\{X(\A); f^\A_\B : \A,\B< \W_1\}$} of metric continua and retractions. This section shows how to construct each $X(\B+1)$ from $X(\B)$. The following section deals with limit ordinals.
	
	To begin let $X(0)$ be the $\sin(1/x)$ continuum. Write $a_0$ for the endpoint and select a sequence $q^n_0$ in $X(0)$ with $x$-coordinates strictly increasing to $0$ from below.
	Observe $q^n_0$ satisfies the following definition of being a \textit{thick half-tail}.
	
	\begin{definition}\label{halftail}
		For $X$ a continuum we define a \textit{half-tail} at $a \in X$ as a sequence $q^n \subset X$ with the properties:
		
		\begin{enumerate}
			\item $[a,q^1] \varsubsetneq [a,q^2] \varsubsetneq \ldots $
			\item $\bigcup \big \{[a,q^n]:n \in \NN \big \} = \K(a)$
			\item For each $x \in X$ and $n \in \NN$ either $[a,q^n] \subset [a,x]$ or $[a,x] \subset [a,q^n]$.
		\end{enumerate}
		
		Moreover the half-tail $q^n$ is called \textit{thick} to mean each $[a,q^n]$ is the closure of its interior.
	\end{definition}
	
	Next use induction to find a family $\cD_0 = \{D_0^1,D_0^2, \ldots \}$
	of pairwise disjoint subsets of $\K(a_0)$ with each $D^n_0 \subset [a_0,q^n_0]$ dense.
	The fact that $q^n_0$ is thick implies that $\cD_0$ is a \textit{tailing family} as defined below.

	\begin{definition}\label{tailingfamily}
		Suppose the continuum $X$ has a half-tail $q^n$ at $a \in X$. By a \textit{tailing family} on $X$
		we mean a pairwise disjoint collection $\cD = \{D^1,D^2, \ldots\}$ of countable subsets of $X$
		with each $D^n \subset [a,q^n]$ and $D^n \cap [a,q^m]$ dense in $[a,q^m]$ for each $m \le n$.
	\end{definition}

	The notions of a half-tail and tailing family are pivotal to our example. Indeed as part of the construction we will at stage $\A<\W_1$ choose a half-tail $q^n_\A$ at $a_\A \in X(\A)$ and a tailing family $\cD_\A$ on $X(\A)$ so both objectss behave nicely with respect to the bonding maps. This is made precise below.
	
	In the next definition and throughout when we write for example \mbox{$a_\B \mapsto a_\G$} the map in question is understood to be the bonding map $f^\B_\G$.
	Similarily for subsets $B \subset X(\B)$ and $C \subset X(\G)$ we write $B \to C$ to mean $f^\B_\G(B) = C$.

	\begin{definition}\label{deftailmap}
		Suppose the continua $X(\B)$ and $X(\G)$ have  half-tails $q^n_\B$ and $q^n_\G$ at $a_\B \in X(\B)$ and $a_\G \in X(\G)$ and tailing families $\cD_\B$ and $\cD_\G$ respectively. The map \mbox{$f:X(\B) \to X(\G)$} is called \textit{coherent} to mean $a_\B \mapsto a_\G$ and $[a_\B,q^n_\B ] \to [a_\G,r^n_\G]$ and $f$ induces bijections $D^n_\B \to D^n_\G$ for each $n \in \NN$. The system $\{X(\B); f^\B_\G: \G,\B < \A\}$ is called coherent to mean each bonding map is coherent.
	\end{definition}
	
	At stage $\A < \W_1$ we have already constructed the coherent inverse system \mbox{$\{X(\B);f^\B_\G: \B,\G < \A\}$} of hereditarily unicoherent metric continua and retractions. We assume the following objects have been specified for each $\B < \A$:

	\begin{enumerate}[label=(\roman*)]
		\item Half-tails $q^n_\B$ at $a_\B \in X(\B)$
		
		\item Tailing families $\cD_\B = \{D_\B^1, D_\B^2, \ldots\}$ on $X(\B)$
	\end{enumerate}
	
	We also assume for each $\G,\D < \A$ the two conditions hold:
	
	\begin{enumerate}[label=(\alph*)]
		\item $\bigcup \{X(\D): \D < \G\} \subset X(\G) - \K(a_\G)$ 
		\item $\bigcup \big \{f^\G_\D  \big (X(\G) - X(\D) \big ): \A > \G > \D \big \} = \K(a_\D)$
	\end{enumerate}

	Conditions (a) and (b) come straight from \cite{one} and will ensure the limit has exactly two composants.
	Coherence will ultimately be used to construct a tailing family on the inverse limit. Once we have shown the limit is a Bellamy continuum, the next lemma gives our main result.

	\begin{lemma}\label{tailingdense}
		Suppose $X$ admits a half-tail $q^n$ and tailing family $\cD= $ $\{D^1,D^2,\ldots\}$.
		Then $X$ is separable.
	\end{lemma}
	
	\begin{proof}
		Clearly $\bigcup \cD = D^1 \cup D^2 \cup \ldots $ has the cardinality of $\NN \times \NN$ which is well known to be countable.
		Now suppose $U \subset X$ is open.
		Since $\K(a)$ is dense it meets $U$.
		Since $\K(a) = \bigcup _n [a,q^n]$ some $[a,q^n]$ meets $U$.
		Since $[a,q^n] \cap U$ is open in $[a,q^n]$ it  contains an element of $D^n$ by the definition of a tailing family.
		We conclude $\bigcup \cD$ is dense in $X$.
	\end{proof}
	
	\noindent
	
	We are now ready to begin the successor step. Suppose $\A=\B+1$ is a successor ordinal. We will construct the hereditarily unicoherent continuum $X(\B+1)$ and retraction $f^{\B+1}_\B:X(\B+1) \to X(\B)$. Then we can define the bonding maps $f^{\B+1}_\G = f^{\B+1}_\B  \circ f^{\B}_\G $.
	We will specify the objects (i) and (ii) when $\B$ is replaced by $\B+1$. Finally we will check the enlarged system is coherent and Conditions (a) and (b) hold for all $\G,\D \le \B+1$.
	
	To begin the construction of $X(\B+1)$ consider the following subset $N$ of $X(\B) \times [0,1]$.
	
	\begin{center}
		$N =  \big( \bigcup \big \{ [a_\B,q^n_\B] \times \{1/(2n-1),1/2n\} : n \in \NN \big \} \big ) \cup \big ( X(\B) \times \{0\} \big )$.
	\end{center}
	
	Define the points $b_0,b_1,b_2,\ldots \in N$.

	\begin{center}
		\begin{tabular}{ c c c }
			$b_{4n} = \big (a_\B, 1/(2n+1) \big )$ &  & $ b_{4n+1} = \big (q^{n}_\B, 1/(2n+1) \big )$ \\ 
			&  & \\  
			$b_{4n+2} = \big (q^{n}_\B, 1/(2n+2) \big )$ &  & $b_{4n+3} = \big (a_\B, 1/(2n+2) \big ) $  
		\end{tabular}
	\end{center}

	To obtain $X(\B+1)$ make for each $n \in \NN$ the identification $b_{4n+1} \sim b_{4n+2}$.
	Call that point $c_{2n+1} \in X(\B+1)$.
	Then make the identification $b_{4n+3} \sim b_{4(n+1)}$.
	Call that point $c_{2(n+1)} \in X(\B+1)$.

	Write $J({2n})$ for the quotient space of each $[a_\B,q^n_\B] \times \{1/(2n+1)\}$
	and $J({2n+1})$ for the quotient space of $[a_\B,q^n_\B] \times \{1/2n\}$. 
	Observe each $J(n)$ is irreducible from $c_{n}$ to $ c_{n+1}$.
	
	Clearly the quotient space of $\bigcup\{J(n): n \in \NN\}$
	is connected and its closure is the union with $X(\B) \times \{0\}$.
	Therefore $X(\B+1)$ is connected.
	Identify $X(\B)$ with the subspace $X(\B) \times \{0\}$.
	The projection $X(\B) \times [0,1] \to X(\B)$ respects the identifications
	and therefore induces a retraction $f^{\B+1}_\B: X(\B+1) \to X(\B)$.
	
	\begin{figure}[!h]
		\centering
		
		\begin{tikzpicture}

		\usetikzlibrary{snakes}

		\begin{scope}[ xscale = -40]


		\draw[ultra thick, magenta] (1/25,0) -- (1/25,8);
		
		\foreach \n in {2,3,4,5,6}
		{
			\draw (1/2/\n,0) -- (1/2/\n,8-8/\n);
			\draw ({1/(2*\n-1)},0) -- ({1/(2*\n-1)},8-8/\n);
			
		}

		\foreach \n in {2,3,4,5,6}
		{
			\draw[ dashed] (1/2/\n,8-8/\n) -- ({1/(2*\n-1)},8-8/\n);
			\draw[dashed] ({1/(2*\n)},0) -- ({1/(2*\n+1)},0);
		}
		
		\draw[dotted] (1/20 ,4) -- (1/13,4);

		\node[circle, fill = orange, minimum width = 0.15cm, inner sep = 0cm] at (1/3,0) {};
		\node[below] at (1/3,0-0.25) {$a_{\B+1} = (a_\B,1)$};
		\node[circle, fill = orange, minimum width = 0.15cm, inner sep = 0cm] at (1/3,4) {};
		\node[above] at (1/3,4+0.25) {$(q^1_\B,1)$};
		
		\node[circle, fill = orange, minimum width = 0.15cm, inner sep = 0cm] at (1/4,0) {};
		\node[below] at (1/4,0-0.25) {$(a_\B,1/2)$};
		\node[circle, fill = orange, minimum width = 0.15cm, inner sep = 0cm] at (1/4,4) {};
		\node[above] at (1/4,4+0.25) {$(q^1_\B,1/2)$};
		
		\node[circle, fill = orange, minimum width = 0.15cm, inner sep = 0cm] at (1/5,0) {};
		\node[below] at (1/5,0-0.25) {$(a_\B,1/3)$};
		\node[circle, fill = orange, minimum width = 0.15cm, inner sep = 0cm] at (1/5,8-8/3) {};
		\node[above] at (1/5,8-8/3+0.25) {$(q^2_\B,1/3)$};
		\node[circle, fill = orange, minimum width = 0.15cm, inner sep = 0cm] at (1/6,8-8/3) {};

		\node[circle, fill = orange, minimum width = 0.15cm, inner sep = 0cm] at (1/25,4) {};
		\node[right] at (1/30,4+0.25) {$q^1_\B$};
		
		\node[circle, fill = orange, minimum width = 0.15cm, inner sep = 0cm] at (1/25,8-8/3) {};
		\node[right] at (1/30,8-8/3+0.25) {$q^2_\B$};

		\node[circle, fill = orange, minimum width = 0.15cm, inner sep = 0cm] at (1/25,0) {};
		\node[right] at (1/30,0) {$a_\B$};

		\node[left] at (1/3 + 0.005,2) {$J(0)$};
		
		\node[left]  at (1/4 + 0.005,2) {$J(1)$};
		
		\node[left]  at (1/5 + 0.0025,4-4/3) {$J(2)$};
		
		\node[left]  at (1/6 ,4-4/3) {$J(3)$};
		
		\node[above] at (1/25,8.25) {$X(\B)$};

		\end{scope}
		
		\end{tikzpicture}
		
		\vspace{0.5cm}
		
		\caption{
			\openup 0.75em Schematic for $X(\B+1)$. Dashed lines indicate identifications.
			The bonding map $f^{\B+1}_\B$ projects to the right.}\label{SeparablyBellamySuccessor1Input}
	\end{figure}
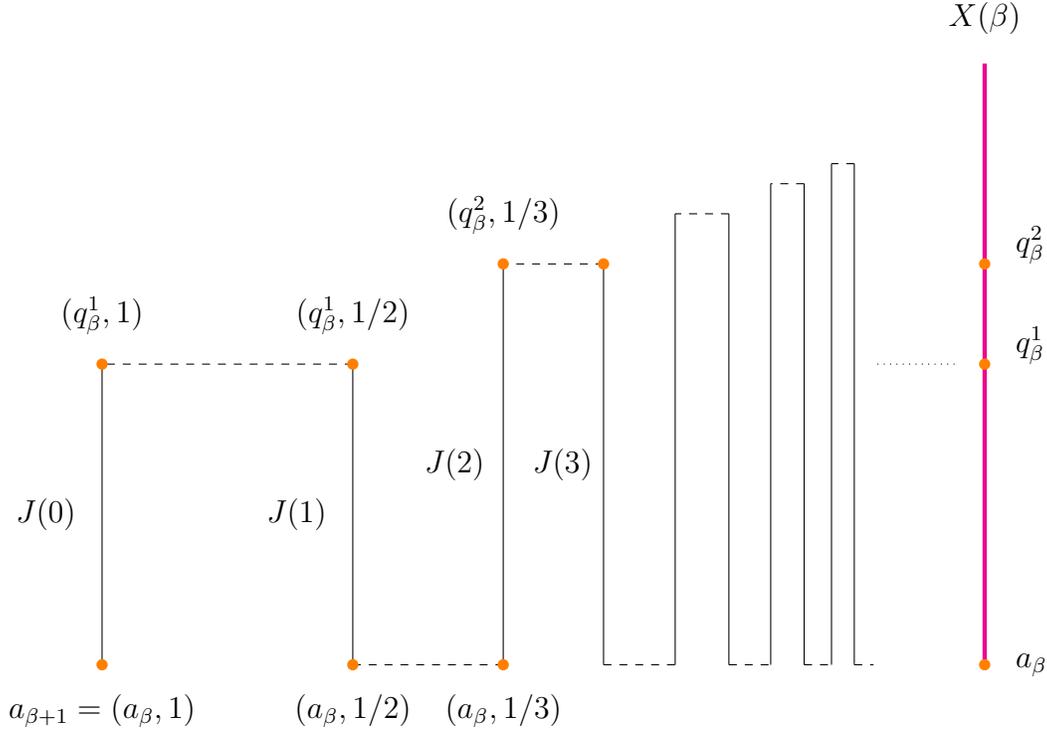

	\begin{claim5}
		$X(\B+1)$ is compact metric.
	\end{claim5}
	
	\begin{proof}
		Let $\cP$ be the partition induced on $N$ by the identifications.
		Each partition element is either a singleton or doubleton hence closed.
		Let $P \in \cP$ be contained in the open $U \subset N$.
		We claim some open neighborhood $W$ of $P$ is a union of partition elements.
		Then $V = U \cap W$ witnesses how $\cP$ is upper semicontinuous and $X(\B+1)$ is compact metric by \cite{nadlerbook} Lemma 3.2.
		
		For $P$ a singleton of some $J(2n)$ take $W = J(2n) - \{b_{4n},b_{4n+1}\}$.
		For $P$ a singleton of some $J(2n+1)$ take $W = J(2n+1) - \{b_{4n+2},b_{4n+3}\}$.
		For $P$ a singleton of $X(\B)$ observe $U \cap X(\B)$ is an open neighborhood of $P$ in $X(\B)$.
		Since $F = f^{\B+1}_\B$ respects $\cP$ the open set $W = F^{-1} \big (U \cap X(\B) \big )$ is a union of partition elements.
		
		For $P$ a doubleton without loss of generality some $n \in \NN$ has $P = \big  \{b_{4n+1}, b_{4n+2} \big \}$.
		Since $J(2n) - b_{4n}$ and $J(2n+1) - b_{4n +3}$ are open in $N$ 
		so is the union $W = \big (J(2n) - b_{4n} \big ) \cup \big (J(2n+1) - b_{4n +3} \big )$.
		The case for $P = \big  \{b_{4n+3}, b_{4(n+1)} \big \}$ is similar.
	\end{proof}

	\begin{claim5}\label{5succ(a)}
		The composant $\K(a_{\B+1}) = X(\B+1)-X(\B)$.  Hence Condition (a) holds for all $\G,\D \le \B+1$.
	\end{claim5}
	
	\begin{proof}Let $n \in \NN$ be arbitrary and $U \subset J(n) - \{c_{n},c_{n+1}\}$ open.
		Since $J(n)$ is irreducible from $c_{n}$ to $c_{n+1}$ 
		boundary bumping implies $J(n) - U = A \cup B$ is the disjoint union of two nonempty clopen sets
		that include $c_{n}$ and $c_{n+1}$ respectively.
		It follows
		
		\begin{center}
			$X(\B+1)-U = \big (J(1) \cup \ldots J(n-1) \cup A \big ) \cup \big (B \cup J(n+1) \cup \ldots \cup X(\B) \big)$
		\end{center}
		
		is a disjoint union of two clopen sets that include $a_{\B+1}$ and contain $X(\B)$ respectively.
		
		Now suppose the subcontinuum $K$ connects $a_{\B+1}$ to $X(\B)$. 
		Observe for each $n \in \NN$ the set $O = J(0) \cup J(1) \cup \ldots \cup J(n-1) - c_n$ is clopen in $X(\B+1) - c_n$
		and has $a_{\B+1} \in O$.
		We conclude all $c_n \in K$.

		For $K$ proper it excludes some open $U \subset J(n) - \{c_{n},c_{n+1}\}$.
		Then the two clopen sets from the first paragraph contradict how $K$ is connected.
		We conclude $\K(a_{\B+1})  \subset X(\B+1)-X(\B)$.
		The other inclusion is witnessed by the subcontinua $J(1) \cup J(2) \cup \ldots \cup J(n)$.
	\end{proof}

	Recall each $J(n)$ is hereditarily unicoherent.
	For $n \le m$ it is straightforward to prove by induction each subcontinuum $I(n,m) = $ $J(n) \cup \ldots \cup J(m)$ is hereditarily unicoherent.

	\begin{claim5}
		Each subcontinuum of $\K(a_{\B+1})$ is contained in some $I(n,m)$.
	\end{claim5}
	
	\begin{proof}
		Suppose $K \subset \K(a_{\B+1})$ is a proper subcontinuum.
		Without loss of generality assume $a_{\B+1} \in K$.
		We claim $K$ meets only finitely many $J(n)$.
		For otherwise there is a sequence $n(1),n(2),\ldots$ with $n(i) \to \infty$ and elements $x_i \in K \cap J(n(i))$.
		
		Consider the sequence $f^{\B+1}_\B(x_1), f^{\B+1}_\B(x_2), \ldots$ in $X(\B)$.
		Since $X(\B)$ is compact metric $f^{\B+1}_\B(x_i)$ has a subsequence tending to some $x \in X(\B)$.
		It follows $x_i$ has a subsequence tending to $(x,0)$.
		Since $K$ is closed it includes $(x,0) \in X(\B)$.
		But then Claim 2 says $K = X(\B+1)$ and so $K \not \subset \K(a_{\B+1})$ contrary to assumption.
	\end{proof}

	\begin{claim5}
		$X(\B+1)$ is hereditarily unicoherent.
	\end{claim5}
	
	\begin{proof}
		
		Suppose $K$ and $L$ are proper subcontinua.
		For $K$ and $L$ contained in $\K(a_{\B+1})$ Claim 3 says $K \cup L \subset I(n,m)$ for some $n \le m$. 
		Since $I(n,m)$ is hereditarily unicoherent $K \cap L$ is empty or connected.
		For $K ,L \subset X(\B)$ then $K \cap L$ is empty or connected since $X(\B)$ is hereditarily unicoherent. 
		
		Now suppose $K$ meets $\K(a_{\B+1})$ and $X(\B+1)-\K(a_{\B+1}) = X(\B)$.
		We claim $K =  K' \cup J(n+1) \cup \ldots \cup X(\B)$ 
		for some $n \in \NN$ and subcontinuum $K' \subset J(n)$ with $c_{n+1} \in K'$.
		To that end let $n \in \NN$ be minimal with $K \cap J(n) \ne \0$.
		Then $K \cap J(n-1) = \0$ and so $c_n \notin K$.
		Observe the subcontinuum $I(0,n) \cup K$ connects $a_{\B+1}$ to $X(\B)$.
		Hence $I(0,n) \cup K=X(\B+1)$ by Claim 2 and $K$ contains $J(n+1) \cup \ldots \cup X(\B)$.
		Since $O = J(0) \cup J(1) \cup \ldots \cup J(n) - c_{n+1}$ is clopen in $X(\B+1) - c_{n+1}$ and disjoint from $X(\B)$
		we must have $c_{n+1} \in K$.
		
		Suppose $K \cap J(n) = A \cup B$ is the disjoint union of two nonempty closed sets with $c_{n+1} \in A$.
		Since $K \cap J(n)$ is closed in $X(\B+1)$ so are $A$ and $B$.
		Observe $C = J(n+1) \cup J(n+2) \cup \ldots \cup X(\B)$ is closed and  $C \cap J(n) = \{c_{n+1}\}$.
		Thus $K = \big ( (C \cap K) \cup A \big ) \cup B$ is a disjoint union of nonempty closed sets.
		We conclude $K \cap J(n)$ is connected.
		
		Thus $K =  K' \cup J(n+1) \cup \ldots \cup X(\B)$ and  $L =  L' \cup J(m+1) \cup \ldots \cup X(\B)$
		for some $m,n \in \NN$ and subcontinua $K' \subset J(n)$ and $L' \subset J(m)$.
		For $m \le n$ we have $K \cap L = L$ is connected.
		For $n \le m$ we have $K \cap L = K$ is connected.
		For $n=m$ we have $K \cap L =  (L'\cap K') \cup J(m+1) \cup \ldots \cup X(\B)$ which is a subcontinuum by hereditary unicoherence of $J(n)$.
	\end{proof}

	\begin{claim5}
		Condition (b) holds for all $\G,\D \le \B+1$.
	\end{claim5}
	
	\begin{proof}

		By induction each $\bigcup \big \{f^\G_\D  \big (X(\G) - X(\D) \big ): \B+1 > \G > \D \big \} = \K(a_\D)$.
		Thus we can factor the set $\bigcup \big \{f^\G_\D  \big (X(\G) - $ $X(\D) \big ): \B + 1 \ge \G > \D \big \}$ as the union $f^{\B+1}_\D  \big (X(\B+1) - X(\D) \big ) \cup \K(a_\D)$. So it is enough to show the second factor is contained in $ \K(a_\D) $. To that end write

		
		\begin{center}
			
			$f^{\B+1}_\D  \big (X(\B+1) - X(\D) \big ) = f^{\B+1}_\D  \big (X(\B+1) - X(\B) \big ) \cup f^{\B+1}_\D  \big (X(\B) - X(\D) \big )$	
			
		\end{center}
		
		By definition $f^{\B+1}_\D = f^\B_\D \circ f^{\B+1}_\B$. Since $f^{\B+1}_\B$ is a retraction the second term equals $f^{\B}_\D  \big (X(\B) - X(\D) \big )$ which is contained in $\K(a_\D)$ by Condition (b) for earlier stages.
		
		Claim \ref{5succ(a)} says $X(\B+1) - X(\B) = \K(a_{\B+1})$. Since $q^n_{\B+1}$ is a half-tail we have $\K(a_{\B+1}) = \bigcup _n [a_{\B+1},q^n_{\B+1}]$ by Property (2). Hence the first term can be written
		$f^{\B+1}_\D \big (\K(a_{\B+1}) \big ) = f^{\B+1}_\D \big ( \bigcup_n[a_{\B+1}, q^n_{\B+1}] \big ) = \bigcup_n f^{\B+1}_\D \big ( [a_{\B+1}, q^n_{\B+1}] \big )$ which equals $\bigcup_n[a_{\D}, q^n_{\D}] = \K(a_\D)$ by coherence of the bonding maps and Property (2) at stage $\D$ respectively.
	\end{proof}

	\begin{claim5}
		The sequence $q^n_{\B+1} = c_{2n-1}$ is a half-tail at $a_{\B+1}$.
	\end{claim5}
	
	\begin{proof}
		
		It is straightforward to verify $J(0) \cup J(1) \cup \ldots \cup J(2n-2)$ is a subcontinuum irreducible from $a_{\B+1}$ to $c_{2n-1}$.
		By hereditary unicoherence that subcontinuum is $[a_{\B+1}, c_{2n-1}]$.
		Since each $c_{2n+1} \notin [a_{\B+1}, c_{2n-1}]$ we have $q^{n+1}_{\B+1} \notin [a_{\B+1}, q^n_{\B+1}]$ which is Property (1).
		To prove Property (2) observe \mbox{$\bigcup_n [a_{\B+1}, q^n_{\B+1}] = $ $J(1) \cup J(2) \cup \ldots = \K(a_{\B+1})$} by Claim 2.

		To prove Property (3) suppose $x \in X(\B+1) -  [a_{\B+1}, q^n_{\B+1}]$.
		Observe $[a_{\B+1}, c_{2n-1}] - c_{2n-1}$ is clopen in $X(\B+1) - c_{2n-1}$.
		Thus the continuum $[a_{\B+1},x]$ includes the point $c_{2n-1}$.
		By Zorn's lemma the continuum $[a_{\B+1},x]$ contains a subcontinuum irreducible from $a_{\B+1}$ to $c_{2n-1}$.
		Since $[a_{\B+1}, c_{2n-1}]$ is the only such subcontinuum 
		we have $[a_{\B+1}, c_{2n-1}] \subset [a_{\B+1},x]$ and so $ [a_{\B+1}, q^n_{\B+1}] \subset [a_{\B+1},x]$.
	\end{proof}
	
	For the successor continua in our system we need  to use how the half-tail is thick. This holds only for successor stages. We will later see the limit continua are indecomposable. Hence their every subcontinuum has void interior and they cannnot admit a thick half-tail.
	
	\begin{claim5}\label{thick}
		The half-tail $q^n_{\B+1}$ is thick. 
	\end{claim5}
	
	\begin{proof}
		We claim each $J(n)^\circ$ contains the dense open subset $J(n) - \{c_n,c_{n+1}\}$.
		Hence $[a_{\B+1}, q^n_{\B+1}] ^\circ$ contains the dense subset $J(1)^\circ \cup J(2)^\circ \cup \ldots \cup J(2n-2)^\circ$.
		Let $x \in J(n) - \{c_n,c_{n+1}\}$ be arbitrary.
		Recall $J(n)$ is a copy of some $[a_\B,q^m_\B]$ with $x$ corresponding to some $y \in [a_\B,q^m_\B]$
		and each of $c_n,c_{n+1}$ corresponding to exactly one of $a_\B$ or $q^m_\B$.
		
		Choose open $U \subset X(\B)$ with $y \in U \subset X(\B) - \{a_\B,q^m_\B\}$ 
		and positive $\displaystyle \E < \min \Big \{ \Big |\frac{1}{m} - \frac{1}{m+1} \Big |, \Big |\frac{1}{m} - \frac{1}{m-1} \Big | \Big \}$.  
		Observe $\displaystyle U \times \Big (\frac{1}{m}  - \E, \frac{1}{m}  + \E \Big )$ is an open subset of $N$ is disjoint from $\{b_0,b_1,b_2,\ldots\}$.
		Hence it corresponds to an open neighborhood of $x$ in $X(\B+1)$.
		The choice of $\E$ ensures it is contained in \mbox{$J(n) - \{c_n,c_{n+1}\}$}.
	\end{proof}
	
	\begin{claim5}
		There exists a tailing family $\cD_{\B+1} = \{D_{\B+1}^1, D_{\B+1}^2 , \ldots \}$ on $X(\B+1)$ that makes $f^{\B+1}_\B$ coherent.
	\end{claim5}
	
	\begin{proof}
		
		We first construct the sets $D^M_{\B+1}$.
		For now let $M \in \NN$ be fixed.
		Since $X(\B+1)$ is metric there is a countable basis $U_1,U_2 \ldots$ for $[a_{\B+1},q_{\B+1}^M]$.
		By the first paragraph $U_1$ meets some $J(n) \subset [b,r^M]$ that maps homeomorphically onto $[a,q^m]$ for some $m \le M$.
		Since $U_1 \cap J(n)$ is open in $J(n)$ the image $f^{\B+1}_\B(U_1) \cap [a_\B,q^n_\B]$ is open in $[a_\B,q^n_\B]$.
		Since $m \le M$ the definition of a tailing family says $D^M_\B \cap [a_\B,q^m_\B]$ is dense in $[a_\B,q^m_\B]$.
		Thus $f^{\B+1}_\B(U_1) \cap [a_\B,q^m_\B]$ contains infinitely many elements of $D^M_\B \cap [a_\B,q^m_\B]$.
		Choose one such $d(1) \in D^M_\B \cap [a_\B,q^n_\B]$ and select $c(1) \in U_1$ with $c(1)  \mapsto d(1)$.

		Proceed by induction.
		At stage $r$ we have chosen distinct $c(i) \in U_i$ for $i=1,2,\ldots, r-1$ and $d(i) = f\big ((c(i) \big)$ are distinct.
		Just like before $f^{\B+1}_\B(U_r)$ includes infinitely many elements of $D^M_\B$. 
		Select some $d(r) \in f(U_r) \cap D^M_\B$ with $d(1),d(2), \ldots , d(r)$ distinct.
		Then select $c(r) \in U_r$ with $c(r)  \mapsto d(r) $.
		
		By construction we get a countable dense subset $E^M_{\B+1} = \{c(1),c(2),\ldots\}$ of distinct elements of $[a_{\B+1},q_{\B+1}^M]$.
		For each $d \in D^M - \{d(1),d(2),\ldots\}$
		use surjectivity to select $c \in [a_{\B+1},q^M_{\B+1}]$ with $c \mapsto d$. Adjoin all such $d$ to $E^M_{\B+1}$ to get the set $D^M_{\B+1}$.
		By construction $f^{\B+1}_\B$ induces a bijection \mbox{$D^M_{\B+1} \to D^M_\B$}.
		
		Now let $M$ vary over $\NN$.
		We get a family $\cD_{\B+1} = \{D^1_{\B+1},D^2_{\B+1}, \ldots\}$ of countable subsets of $X(\B+1)$ 
		with each $f^{\B+1}_\B:D^M_{\B+1} \to D^M_\B$ a bijection and $D^M_{\B+1} \subset [a_{\B+1},q_{\B+1}^M]$ dense.
		Since the elements of $\cD_\B$ are pairwise disjoint so are the elements of $\cD_{\B+1}$. 
		Claim \ref{thick} shows each $[a_{\B+1},q^m_{\B+1}]$ is the closure of its interior. From this it follows $D^M_{\B+1} \cap [a_{\B+1},q^m_{\B+1}]$ is dense in $ [a_{\B+1},q^m_{\B+1}]$ whenever $m \le M$. We conclude $\cD_{\B+1}$ is a tailing family.
		
		To show coherence first recall by construction $J(0)$ is a copy of $[a_\B,q^1_\B]$ and projects under $f^{\B+1}_\B$ onto that copy. Since $a_{\B+1} \in J(0)$ corresponds to the point $a_\B \in [a_\B,q^1_\B]$ we have $a_{\B+1} \mapsto a_\B$. 
		Recall we define $q^n_{\B+1} = c_{2n-1}$ and by definition $c_{2n-1} \sim \big (q^n_\B, 1/(2n-1) \big )$. Since $f^{\B+1}_\B$ is induced by the projection $X(\B) \times [0,1] \to X(\B)$ we have $q^n_{\B+1} \mapsto q^n_{\B} $. 
		Finally recall $[a_{\B+1}, c_{2n-1}] = $ $J(0) \cup J(1) \cup \ldots \cup J(2n-2)$. By cooherence this set maps onto $[a_\B,q^1_\B] \cup [a_\B,q^2_\B] \cup \ldots \cup [a_\B,q^n_\B]$. By Property (1) for $\B$ the image equals $[a_\B,q^n_\B]$ as required.
	\end{proof}

	\section{The Limit Stage}\label{5.4}
	
	\noindent
	This section deals with the limit stage of our construction. Henceforth assume $\A \le \W_1$ is a limit ordinal and \mbox{$\{X(\B); f^\B_\G : \B,\G< \A\}$} a coherent system of hereditarily unicoherent metric continua. For all $\B,\G,\D < \A$ we assume the objects (i) and (ii) from Section \ref{5.4} have been specified and Conditions (a) and (b) hold.
	
	We define  $X(\A) = \varprojlim \{X(\B); f^\B_\G\}$ and each $f^\A_\B$ as the projection from the inverse limit onto its factors. For each $\G<\A$ we identify  $X(\G)$ with the subset $\big \{x\in X(\A): x_\B = x_\G$ for all $\B > \G \big \}$ of $X(\A)$. 
	
	That $X(\A)$ is hereditarily unicoherent follows from a straightforward modification of \cite{HULimits} Corollary 1. To see $X(\A)$ is metric we observe by definition $\A$ is a countable ordinal. 
	The product $\prod_{\B < \A} X(\B)$ of countably many metric spaces is itself a metric space. The inverse limit $X(\A)$ is by definition a subset of that product and therefore a metric space 
	
	It remains to show the enlarged system is coherent; to check Conditions (a) and (b) hold for the enlarged system; and to specify the data (i) and (ii) for $X(\A)$. By coherence at earlier stages each $a_\B \mapsto a_\G$ and $q^n_\B \mapsto q^{n}_\G$. Hence there are well defined points $a_\A = (a_\B)_{\B<\A}$ and $q^n_\A = (q^n_\B)_{\B<\A}$ in the limit $X(\A)$ with $a_\A \mapsto a_\B$ and $q^n_\A \mapsto q^n_\B$. For ease of notation write $a$ and $q^n$ instead of $a_\A$ and $q^n_\A$ respectively.
	
	For the proof that $q^n$ is a half-tail we refer to \cite{Me4} where we introduce the more complicated notion of a \textit{tail} and study inverse limits of tails. The proof for half-tails follows from a close reading of the proofs of \cite{Me4} Claims \ref{longlim4} and \ref{longlim6} and Lemma \ref{betlim}. Hence we have the next claim.

	\begin{claim5}\label{Lht}
		The sequence $\displaystyle q^n$ is a half-tail at $a$. Moreover for each $n \in \NN$ we have $[a,q^n] = \varprojlim \big \{[a_\B,q^n_\B] ; f^\B_\G \big \}$.
	\end{claim5}
	
	Recall the proper subcontinuum $R \subset X$ is called a \textit{rung} to mean each other subcontinuum $K \subset X$ 
	is either disjoint from, contained in, or contains $R$. By a \textit{ladder} on $X$ we mean a nested collection of rungs of $X$ with dense union.
	The proof that $X(\A)$ is indecomposable will follow from the existence of a ladder.

	\begin{lemma}\label{5thm3}
		Suppose $X$ admits a ladder.
		Then $X$ is indecomposable.
	\end{lemma}
	
	\begin{proof}
		
		We first show rungs have void interior.
		Suppose otherwise the rung $R \subset X$ has $R^\circ \ne \0$.
		Let $b \in X-R$ be arbitrary and $C$ the component of $b$ in $X-R$.
		Then $b$ witnesses how $\overline C \not \subset R$.
		At the same time $C \subset X-R$ which implies $ \overline C \subset X - R^\circ $ which in turn implies $R \not \subset \overline C$.
		Since $R$ is a rung the subcontinua $\overline C$ and $R$ must be disjoint.
		But this contradicts boundary bumping which says says $\overline{C}$ meets $R$.
		We conclude each $R^\circ = \0$.
		
		Now suppose $\cL$ is a ladder on $X$
		and the proper subcontinuum $L \subset X$ has nonvoid interior.
		Since $\bigcup \cL$ is dense some rung $R \in \cL$ meets $L^\circ$ and $X-L$ hence contains $L$.
		Since $R$ has void interior so does $L$.
		We conclude each subcontinuum of $X$ has void interior.
		This is equivalent to $X$ being indecomposable.
	\end{proof}
	
	Since our simplest example of a rung is the limiting arc of the $\sin(1/x)$ continuum, and each $X(\B+1)$ looks like a $\sin(1/x)$ continuum limiting to $X(\B)$, the next claim should be unsurprising. 
	
	\begin{claim5}
		Each $X(\G)$ is a rung of $X(\A)$.
	\end{claim5}

	\begin{proof} 	
		The proof uses transfinite induction. 
		Suppose for some $\wt \A \le \A$ and all $\G \le \B< \wt\A$ that $X(\G) \subset X(\B)$ is a rung.
		Now suppose the continuum $K \subset X(\wt \A)$ meets $X(\G)$ and $X(\wt \A)- X(\G)$. 
		
		For $\wt \A = \wt \B + 1$ a successor ordinal there are two possibilities. First that $K \subset X(\wt \B)$. In that case $K \subset X(\wt \B)$ meets $X(\G)$ and $X(\wt \B) - X(\G)$ and we have $X(\G) \subset K$ since $X(\G) \subset X(\wt \B)$ is a rung by induction. Second that $K$ meets $ X(\wt \B+1) - X(\wt \B)$. It follows from Claim \ref{5succ(a)} that $X(\wt \B) \subset X(\wt \B+1)$ is a rung. Hence $X(\wt \B) \subset K$ and $X(\G) \subset K$ by Property (a) at stage $\wt \B$.
		
		For $\wt \A$ a limit ordinal there are again two possibilities. First that $K \subset X(\B)$ for some $\B < \wt \A$. In that case $X(\G) \subset X(\B)$ is a rung by induction hence $X(\G) \subset K$ as required. Second that $K$ is contained in no $X(\B)$. In that case recall $K = \varprojlim \{f^{\wt \A}_\D(K); f^\B_\D: \B,\D < \wt \A\}$. 
		
		If there was $\B < \wt \A$ with $f^{\wt \A}_\D(K) \subset X(\B)$ for all $\D > \B$ we would have \mbox{$K \subset X(\B)$} contrary to assumption.
		Thus for each $\B < \wt \A$ there is $\D > \B$ with $f^{\wt \A}_\D(K) \not \subset X(\B)$. By induction $X(\B) \subset X(\D)$ is a rung. Hence the subcontinuum $f^{\wt \A}_\D(K)$ of $X(\D)$ contains $X(\B)$ and its subset $X(\G)$. By commutativity $f^{\wt \A}_{\D'}(K) =$ $f^{\D}_{\D'} \circ f^{\wt \A}_{\D} (K)$ also contains $X(\G)$ whenever $\G \le \D' \le \D$. In particular whenever $\G \le \D' \le \B$. Since $\B < \wt \A$ is arbitrary we get $X(\G) \subset f^{\wt \A}_{\D'}(K)$ for all $\G \le \D' < \wt \A$. It follows $X(\G) \subset K$ as required.
	\end{proof}

	\begin{claim5}\label{limind}
		$X(\A)$ is indecomposable.
	\end{claim5}
	
	\begin{proof}
		Recall we define $X(\G) = \big \{(x_\B) \in X(\A): x_\B=x_\G \ \All \B > \G \big \}$. 
		By Lemma \ref{5thm3} it is enough to show $\{X(\B): \B < \A\}$ is a ladder in $X(\A)$.
		To that end let $\G < \A$ be arbitrary and $U \subset X(\G)$ open. We must show $\pi_\G^{-1}(U)$ meets some $X(\B)$.
		Let $x_\G \in U$ be arbitrary and consider the following $(x_\B) \in X(\A)$.
		For $\B \ge \G$ define $x_\B = x_\G$.
		We have $f^\B_\G(x_\B)=x_\G$ since the bonding map is a retraction.
		For $\B \le \G$ define $x_\B = f^\G_\B(x_\G)$.
		It follows $(x_\B)$ is a well defined element of $X(\A)$.
		By definition $(x_\A) \in X(\G) \cap \pi_\G^{-1}(U)$.
	\end{proof}

	\begin{claim5}\label{5a}
		Condition (a) holds for all $\G,\D \le \A$.
	\end{claim5}
	
	\begin{proof}
		
		By induction we only need to prove the case $\G = \A$.
		We must show $\K(a_\A)$ is disjoint from each $X(\B)$.
		Since $q^n$ is a half-tail at $a$ 
		each $x \in \K(a)$ is an element of some $[a,q^n] = \varprojlim [a_\B,q^n_\B]$.
		That means $x_{\B+1} \in [a_{\B+1},q^n_{\B+1}]$ for each $\B < \A$.
		Since $q^n_{\B+1}$ is a half-tail at $a_{\B+1}$ we have $[a_{\B+1},q^n_{\B+1}] \subset \K(a_{\B+1})$
		which is disjoint from $X(\B)$ by Condition (a) for $\B+1$.
		We conclude $x_{\B+1} \in  X(\B+1) - X(\B)$. Since $x_\B \in X(\B)$ we have $x_{\B+1} \ne x_\B$ hence $x \notin X(\B)$ by definition of the embedding $X(\B) \to X(\A)$.
	\end{proof}

	\begin{claim5}\label{5b}
		Condition (b) holds for all $\G,\D \le \A$.
	\end{claim5}
	
	\begin{proof}
		By induction and commutativity it is enough to show $f^\A_\D \big (X(\A)-X(\D) \big ) = \K(a_\D)$ for each $\D \le \A$. Claim \ref{5a} says each $X(\B)$ is disjoint from $\K(a_\A)$.
		Hence $X(\A) - X(\B)$ contains $\K(a_\A)$ which maps onto $\K(a_\B)$ by coherence.
	\end{proof}

	\begin{claim5}\label{limtail}
		There exists a tailing family $\cD_{\A} = \{D_{\A}^1, D_{\A}^2 , \ldots \}$ on $X(\A)$ that makes $f^{\A}_\B$ coherent for each $\B < \A$.
	\end{claim5}
	
	\begin{proof}
		It is clear that $a \mapsto a_\B$ and $q^n \mapsto q^n_\B$. Claim \ref{Lht} says each $[a,q^n] = $ $\varprojlim \big \{[a_\B,q^n_\B] ; f^\B_\G \big \}$ hence $[a,q^n] \mapsto [a_\B,q^n_\B]$.
		
		By induction each $f^\B_\G$ is coherent hence induces bijections $D_\B^n \to D_\G^n$. That means the inverse limit $D^n_\A = \varprojlim \big \{D_\B^n; f^\B_\G \big \}$ is a well defined countable subset of $X(\A)$. Lemma 2.5.9 of \cite{Engelking} says the restrictions $f^\A_\B: D^n_\A \to D^n_\B$ are bijective. Claim \ref{Lht} says  each $D^n_\A \subset [a,q^n]$.
		
		To see $D^1_\A, D^2_\A,\ldots$ are pairwise disjoint take $x \in D^n_\A$ and $y \in D^m_\A$ for some $m \ne n$. By definition $f^\A_0(x) \in D_0^n$ and $f^\A_0(y) \in D_0^m$. Since $\cD_0$ is a tailing family $D_0^m$ and $D_0^n$ are disjoint and the result follows.
		
		To show $D_\A = \{D^1_\A, D^2_\A, \ldots\}$ is a tailing family it remains to prove each \mbox{$D_\A^n \cap [a,q^m]$} is dense in $[a,q^m]$ whenever $m \le n$. To that end Claim \ref{thick} says the half-tail $q^n_{\B}$ is thick whenever $\B < \A$ is a successor ordinal. Since $\A$ is a limit ordinal the set $\Gamma = \{\B < \A: \B$ is a successor ordinal$\}$ is cofinal in $\A$. Hence we can write $X(\A) = \varprojlim \{X(\B);f^{\B}_{\G}: \B,\G \in \GG\}$ where each $X(\B)$ comes with a distinguished thick half-tail.
		
		Recall $\pi_\B = f^\A_\B$ and the open subsets of $[a,q^m] = \varprojlim [a_\B,q_\B^m]$ 
		have the form $\pi^{-1}_\B(U)$ for open $U \subset [a_\B,q_\B^m]$.
		By thickness $[a_\B,q_\B^m]^\circ$ is dense in $[a_\B,q_\B^m]$.
		Hence $V = U \cap [a_\B,q_\B^m]^\circ$ is a nonempty open subset of $X(\B)$.
		Since $[a_\B,q_\B^m] \subset [a_\B,q_\B^n]$ we see $V$ is open in $[a_\B,q_\B^n]$ as well.
		Since $D_\B^n$ is dense in $[a_\B,q_\B^n]$ there is $d \in V \cap D_\B^n$.
		Since $D_\B^n = \pi_\B(D^n_\A)$ we have $\pi_\B(x) = d$ for some $x \in D^n_\A$. But then $\pi_\B(x) \in V \subset U$ and so $x \in \pi_\B^{-1}(U)$ as required.
	\end{proof}
	
	Finally we have the main example.
	
	\begin{thm5} \label{5thm5}
		There exists a separable Bellamy continuum.
	\end{thm5}
	
	\begin{proof}
		
		By transfinite recursion we have a coherent system of metric continua $\{X(\A);f^\A_\B: \B,\A < \W_1\}$. Let $X$ be the inverse limit. Recall this section assumes $\A\le \W_1$ is a limit ordinal. For the special case $\A= \W_1$ we have $X = X(\A)$ and Claim \ref{limtail} says $X$ admits a tailing family. Lemma \ref{tailingdense} then says $X$ is separable.
		
		Once we have expressed $X$ as the inverse limit of a system of indecomposable metric continua, Conditions (a) and (b) and Theorem 1 of \cite{one} will say $X$ is indecomposable with at most two composants.
		The \textit{trivial composant} $E \subset X$ is the set $\bigcup \{X(\A): \A < \W_1\}$ of eventually constant $\W_1$- sequences.
		The \textit{nontrivial composant} $X-E$ is equal to the limit \mbox{$\varprojlim \{\K(a_\A);f^\A_\B: \B,\A < \W_1\}$}. In this case the trivial composant contains the point $(a_\A)_{\A<\W_1}$ hence is nonempty by construction.
		
		Claim \ref{limind} says $X(\A)$ is indecomposable whenever $\A< \W_1$ is a limit ordinal. We claim $\GG = \{\A<\W_1: \A$ is a limit ordinal$\}$ is cofinal in $\W_1$. Thus we can write $X = \varprojlim \{ X(\A);f^\A_\B: \A,\B \in \GG\}$ as the inverse limit of a system of indecomposable metric continua.
		
		To that end let $\B<\W_1$ be arbitrary and observe $\B \times \W$ is well-ordered under $(\D,m) \le (\G,n) \iff \big( m < n $ or $m = n$ and $\D\le \G )$. Hence $\B \times \W$ is isomorphic to some ordinal $\wt \B$. It is clear $\B \times \W$ has no top element hence $\wt \B$ is a limit ordinal. Since the initial segment $\B \times \{1\}$ is a copy of $\B$ we have $\B \le \wt \B$ and since $\B \times \W$ is countable we have $\wt \B < \A$. 
		
		We conclude $X$ has exactly two composants.
		Let the Bellamy continuum $\wt X$ be obtained by choosing any $x \in E$ and $y \in X-E$ and identifying $x \sim y$. Since $\wt X$ is the image of $X$ under the quotient map it is separable.
	\end{proof}

	
	\section{Separability and Metrisability}
	
	\noindent
	We make some observations on the global properties of the inverse limit $X$ from Section \ref{5.4}.
	The first is the tailing family $\cD = \{D^1, D^2, \ldots \}$ on $X$ has each $D^n$ contained in the nontrivial composant
	$\varprojlim \{\K(a_\A);f^\A_\B: \B,\A < \W_1\}$.
	Thus we have the following.
	
	\begin{claim5}\label{nontrivsep}
		The nontrivial composant of $X$ is separable.
	\end{claim5}
	
	Thus the nontrivial composant can be considered \textit{small}.
	On the other hand Lemma \ref{nonsep} shows the trivial composant is \textit{large}.
	
	\begin{lemma}\label{nonsep}
		Suppose the topological space $T$ is the union of an $\W_1$-chain of proper closed subsets.
		Then $T$ is non-separable.
	\end{lemma}
	
	\begin{proof}
		Let $\cB = \{B(\A): \A < \W_1 \}$ be a chain of proper closed subsets of $T$.
		Suppose $D = \{d_1,d_2,\ldots\} \subset T$ is dense.
		For each $n \in \NN$ there exists $\A(n) < \W_1$ with $d_n \in B(\A(n))$.
		The countable subset $\{\A(n): n \in \NN\}$ has an upper bound $\A < \W_1$.
		Since $\cB$ is a chain we have $\{d_1,d_2,\ldots\} \subset B(\A)$.
		Since $B(\A)$ is closed and proper $D$ is not dense.
		Hence $T$ is non-separable.
	\end{proof}
	
	The trivial composant of $X$ is the union of the $\W_1$-chain $\{X(\A): \A < \W_1 \}$ of proper subcontinua.
	Thus Lemma \ref{nonsep} gives the next two claims.
	
	\begin{claim5}\label{trivsep}
		The trivial composant of $X$ is non-separable.
	\end{claim5}
	
	\begin{claim5}
		The continuum $X$ is separable but not hereditarily separable and not dense-hereditarily separable. 
	\end{claim5}
	
	This suggests two open problems.
	
	\begin{qn5}
		Is the nontrivial composant of $X$ hereditarily separable or dense-hereditarily separable? 
	\end{qn5}
	
	\begin{qn5}\label{q2}
		Does there exist an hereditarily separable or dense-hereditarily separable Bellamy continuum? 
	\end{qn5}
	
	The Introduction mentions all known Bellamy continua are Stone-\v Cech remainders, 
	or arise from inverse-limits constructions similar to Section \ref{5.4}.
	For a positive answer to Question \ref{q2} we imagine an entirely new class of examples 
	would be needed to obviate Lemma \ref{nonsep}.
	
	Our next result is that neither composant of $X$ is metrisable.
	Lemma \ref{nonmet} is similar to Lemma \ref{nonsep} and applies to the trivial composant.
	
	\begin{lemma}\label{nonmet}
		No metric space is the union of an $\W_1$-chain of compact proper subsets.
	\end{lemma}
	
	\begin{proof}
		First recall \cite{nagata} Theorem IV.5 (F) says compactness and sequential compactness are equivalent for metric spaces.
		Now suppose the metric space $M$ is the union of the chain $\cK = \{K(\A): \A < \W_1 \}$ of compact proper subsets.
		We claim $M$ is compact.
		
		Suppose $x_1,x_2,\ldots$ is an arbitrary sequence in $M$.
		For each $n \in \NN$ there exists $\A(n) < \W_1$ with $x_n \in K(\A(n))$.
		The countable subset $\{\A(n): n \in \NN\}$ has an upper bound $\A < \W_1$.
		Since $\cK$ is a chain we have $\{x_1,x_2,\ldots\} \subset K(\A)$.
		Since $K(\A)$ is compact metric there is a point $x \in K(\A)$ 
		and a subsequence $(y_n)$ of $(x_n)$ with $y_n \to x$.
		We conclude $M$ is sequentially compact hence compact.
		
		Since $M$ is compact metric it admits a countable base $U_1,U_2,\ldots$ of open sets.
		Since $M = \bigcup \cK$ each $U_n \in \NN$ meets $K(\B(n))$ for some $\B(n)< \W_1$. 
		The countable subset $\{\B(n): n \in \NN\}$ has an upper bound $\B < \W_1$.
		Since $K(\B)$ meets all $U_n$ it is dense.
		Since $K(\B)$ is compact and $M$ metric $K(\B)$ is closed.
		Thus $K(\B)=M$ contradicting the assumption that each $K(\B)$ is proper.
	\end{proof}
	
	The trivial composant of $X$ is the union of the $\W_1$-chain $\bigcup \{X(\A): \A < \W_1 \}$ of proper subcontinua.
	Thus Lemma \ref{nonmet} gives the next claim.
	
	\begin{claim5}
		The trivial composant of $X$ is non-metrisable.
	\end{claim5}
	
	Next we prove the same for the nontrivial composant.
	
	\begin{claim5}
		The nontrivial composant of $X$ is non-metrisable.
	\end{claim5}
	
	\begin{proof}
		Clearly the nontrivial composant $X-E$ is connected.
		Claim \ref{nontrivsep} says $X-E$ is separable.
		Lemma 3 of \cite{Lipham01} says $X-E$ is strongly indecomposable as defined in \cite{Lipham01} Section 2.
		For a contradiction suppose $X-E$ is metrisable.
		
		Let the $X \to \wt X$ be the quotient map from Theorem \ref{5thm5}.
		Since $X-E$ is homeomorphic to its image we know $\wt X$ is a compactification of $X-E$.
		Then \cite{Lipham01} Theorem 8  says $\wt X$ is irreducible between some pair of points $\{a,b\}$.
		But that means $a$ and $b$ have different composants contradicting how $\wt X$ is a Bellamy continuum.
	\end{proof}
	
	To close the section we remark that metrisability cannot be droped from the hypothesis of Lemma \ref{nonmet}.
	Example 5.8 of \cite{me1} is the closed unit ball in the Hilbert space $\ell^2(\W_1)$ 
	of square-summable functions $\W_1 \to \mathbb R$ under the weak topology. 
	In fact $\ell^2(\W_1)$ is even a continuum and the elements of the $\W_1$-chain are nowhere dense subcontinua.
	
	\section{Embedding Properties}
	
	\noindent
	Bellamy \cite{one} has shown each metric continuum is a retract of a Bellamy continuum.
	We use a similar technique to show each hereditarily unicoherent metric continuum is a retract of a separable Bellamy continuum.
	
	For our example we took $X(0)$ the $\sin(1/x)$ continuum. There is no obstruction to using instead any hereditarily unicoherent metric continuum $Y$.
	In case $Y$ admits thick half-tail $q^n_0$ at the point \mbox{$a_0 \in Y$} we can simply take $Y = X(0)$ as the bottom element of the system $\{X(\A);f^\A_\B: \A,\B< \W_1\}$ from Section \ref{5.3}.
	
	Writing $X$ for the limit we see the projection $\pi_0: X \to X(0)$ is a retraction. 
	Therefore each $\pi_0(x)$ is a point of the trivial composant. Let the separable Bellamy continuum $\wt X$ be obtained by choosing some $x$ in the nontrivial composant and identifying $x \sim \pi_0(x)$.
	Clearly $\pi_0$ induces a retraction $\wt X \to Y$ from a separable Bellamy continuum.
	
	In case $Y$ has no such tail, we must first build an hereditarily unicoherent metric continuum $X(0)$ with a thick half-tail $q^n_0$ at the point \mbox{$a_0 \in X(0)$} and a retraction $X(0) \to Y$. We then take $X(0)$ as the bottom element of the inverse system and compose the retractions $\wt X \to X(0)$ and $X(0) \to Y$ to get the desired retraction onto $\wt X \to Y$.
	It remains to construct such a continuum $X(0)$.
	
	\begin{lemma}\label{X0retract}
		Suppose the metric continuum $Y$ is hereditarily unicoherent.
		There exists a retraction $X(0) \to Y$ from an hereditarily unicoherent metric continuum $X(0)$ 
		with a thick half-tail $q^n_0$ at the point $a_0 \in X(0)$.
	\end{lemma}
	
	\begin{proof}
		Choose a countable dense subset $D = \{d(1),d(2), \ldots\}$ of $Y$.
		For the sequence $s = (1,1,2,2,$ $1,1, 2,2,3,3, \ldots)$ define each $p_n = d(s(n))$.
		Define the closed subset $N \subset Y \times [0,1]$.
		
		\begin{center}
			$\displaystyle N =  \Big( \bigcup_{n \in \NN}  [p_1,p_n] \times \{1/n\}  \Big ) \cup \big ( Y \times \{0\} \big )$.
		\end{center}
		
		To obtain $X(0)$ from $N$ first make for each odd $n \in \NN$ the identification $\big (p_n, 1/n \big ) \sim  \big (p_{n+1}, 1/(n+1) \big )$ .
		Then make for each even $n \in \NN$ the ident- ification $\big (p_1, 1/n \big ) \sim  \big (p_1, 1/(n+1) \big )$ .
		
		The picture for $X(0)$ is similar to Figure \ref{SeparablyBellamySuccessor1Input}.
		Then $X(0)$ is a metric space as a subset of the product $Y \times [0,1]$ of metric spaces. It is straightforward to show the first summand of $N$ is connected with closure $X(0)$. Since $Y \times [0,1]$ is compact so is $X(0)$. Identify $Y$ with the subset $Y \times \{0\}$ of $X(0)$.
		
		A similar argument to Section \ref{5.3} shows $X(0)$ is hereditarily unicoherent and $q^n_0 = (p_n,1/n)$ is a thick half-tail at $a_0 = (p_1,1)$
		Clearly the projection \mbox{$N \to Y$} onto the first coordinate respects the partition.
		Hence it induces a retraction $X(0) \to Y$.
	\end{proof}
	
	The theorem follows.
	
	\begin{thm5}\label{retract}
		Each  hereditarily unicoherent metric continuum is a retract of a separable Bellamy continuum.
	\end{thm5}
	
	There are several directions one might generalise Theorem \ref{retract}. The first is to drop the reference to hereditary unicoherence.
	
	\begin{qn5}
		Is each metric continuum a retract of a separable Bellamy continuum?
	\end{qn5}

	The methods in Sections \ref{5.3} and \ref{5.4} rely heavily on hereditary unicoherence and do not generalise.
	One special case that seems approachable is when $Y$ is arcwise connected. 
	
	If we take the proof of Lemma \ref{X0retract} and replace each $[p_1,p_n]$ with some arc from $p_1$ to $p_n$ the resulting space $X(0)$ is a \textit{spiral} over $Y$. That means the composant $\K(a_0)$ of $Y - X(0)$ of $a_0$ is an open ray. 
	So while $X(0)$ is not itself hereditarily unicoherent, we see any two subcontinua of $\K(a_0)$ have connected intersection.
	Since the maps $f^\A_0$ map $X(\A)-X(0)$ into $\K(a_0)$ it seems likely one could adapt our construction while paying very close attention to where hereditary unicoherence is used, and thus get a retractions from separably Bellamy continua onto arcwise connected continua. In particular \cite{nadlerbook} Theorem 8.23 says this includes all locally connected continua.
	
	One might also try to replace the hypothesis of $Y$ being metrisable with merely being separable.
	
	\begin{qn5} \label{sepretract}
		Is each separable hereditarily unicoherent continuum a retract of a separable Bellamy continuum?
	\end{qn5}
	
	Secions \ref{5.3} and \ref{5.4} do not generalise to answer Question \ref{sepretract}. 
	This is because in the non-metric realm we might encounter the following type of subset.
	
	\begin{definition}
		The dense subset $D$ of the topological space $T$ is called \textit{resolvable} to mean it is the disjoint union of two dense subsets. Otherwise we call the subset \textit{irresolvable}.
	\end{definition}
	
	Irresolvable sets are pathological objects that never occur in metric continua.
	
	\begin{lemma} \label{resolvable}
		Every dense subset of a metric continuum is resolvable.
	\end{lemma}
	
	\begin{proof}
		Suppose $X$ is a metric continuum. We first show each nonvoid open subset $U \subset X$ is uncountable.
		Corollary 5.5 of \cite{nadlerbook} says $U$ contains a proper subcontinuum $K$.
		Then Urysohn's lemma (see \cite{NagataTopology} Theorem III.2) says $K$ surjects onto $[0,1]$ hence is uncountable.
		We conclude $U$ is uncountable.
		
		Now suppose $X$ is metric and $D \subset X$ dense. Since finite subsets are closed we see $D$ is infinite. Let $U_1,U_2,\ldots$ be a countable basis for $X$. Since each $U_i$ is infinite we can use induction to select distinct elements $a_1,b_1,a_2,b_2, \ldots \in X$ with each $a_i,b_i \in U_i$. By construction $A = \{a_1,a_2,\ldots\}$ and $B = \{b_1,b_2,\ldots\}$ are dense. Since $D(2) = D - A$ contains $B$ it is also dense. For $D(1) = A$ we get a disjoint union $D=D(1) \cup D(2)$ into disjoint dense subsets.
	\end{proof}
	
	On the other hand Theorem 4.1 of \cite{Submaximal} shows the construction 
	of a countable irresolvable subset of the non-metric separable continuum $[0,1]^\cn$.
	The following lemma will be useful.

	\begin{lemma} \label{irresolvabledifference}
		Suppose the subsets $D$ and $F$ of the continuum $X$ differ by only finitely many elements.
		Then neither or both of $D$ and $F$ are resolvable.
	\end{lemma}
	
	\begin{proof}
		We claim $D$ is resolvable if and only if $A = D \cup F$ is resolvable.
		First suppose $D = D(1) \cup D(2)$ is a disjoint union of dense subsets.
		For $A(1) = D(1)$ and $A(2) = D(2) \cup (A-D)$ the equality $A = A(1) \cup A(2)$ witnesses how $A$ is resolvable. 
		
		Now suppose $A = A(1) \cup A(2)$ is a disjoint union of dense subsets.
		Clearly $D(1) = A(1) \cap D$ and $D(2) = A(2) \cap D$ are disjoint with union $D$.
		By assumption there are finite subsets $C(1),C(2) \subset X$ with
		$D(1) = A(1) - C(1)$ and $D(2) = A(2) - C(2)$.
		Since nonvoid open subsets of $X$ are infinite we see $C(1)^\circ = C(2)^\circ = \0$.
		Thus $\overline {D(1)} = $ $\overline {A(1) - C(1)} = \overline {A(1)} - C(1)^\circ = \overline {A(1)} = X$.
		
		Hence $D(1)$ is dense and likewise for $D(2)$.
		We conclude $D$ is resolvable.
		The same proof shows $F$ is resolvable if and only if $A = F \cup D$ is resolvable.
		We conclude $D$ is resolvable if and only if $F$ is resolvable. 
	\end{proof}

	Henceforth suppose the Hausdorff continuum $X(\B+1)$ has a thick half-tail $q^n_{\B+1}$ at \mbox{$a_{\B+1} \in X(\B+1)$} 
	and $\cD_{\B+1} = \{D_{\B+1}^1, D_{\B+1}^2, \ldots\}$ is a tailing family on $X(\B+1)$.
	Let the Hausdorff continuum $X(\B+2)$ and thick half-tail $q^n_{\B+2}$ at $a_{\B+2} \in X(\B+2)$ 
	and bonding map $f^{\B+2}_{\B+1}: X(\B+2) \to X(\B+1)$ be as constructed in Section \ref{5.3}.
	The next claim shows the obstruction to choosing an appropriate $\cD_{\B+2}$ on $X(\B+2)$.
	
	\begin{claim5}\label{resolvableStep}
		Suppose $X(\B+2)$ admits a tailing family that makes $f^{\B+2}_{\B+1}$ coherent.
		Then $[a_{\B+1}, q_{\B+1}^1] \cap D^2_{\B+1}$ is resolvable as a subset of $[a_{\B+1}, q_{\B+1}^1]$.
	\end{claim5}
	
	\begin{proof}
		
		Suppose the tailing family $\cD_{\B+2} = \{D_{\B+2}^1, D_{\B+2}^2, \ldots\}$ makes $f^{\B+2}_{\B+1}$ densely coherent.
		Define the subset $F = D^2_{\B+2} - \{c_1,c_2\}$ of $X({\B+2})$.
		Recall the subcontinua $J(0), J(1),J(2) \subset X(\B+2)$ have dense interior and
		
		\begin{center}
			$J(0) \cap J(1) = \{c_1\} \ \ \ \ \ \ \ \ J(1) \cap J(2) = \{c_2\} \ \ \ \ \ \ \ \ J(0) \cap J(2) = \0$.
		\end{center}
		
		Therefore $F \cap J(0)$, $F \cap J(1)$ and $F \cap J(2)$ are pairwise disjoint and dense in $J(0)$, $J(1)$ and $J(2)$ respectively.
		
		Recall the bonding map $f^{\B+2}_{\B+1}$ induces projections 
		$J(0) \to [a_{\B+1}, q_{\B+1}^1]$ and  $J(1) \to [a_{\B+1}, q_{\B+1}^1]$ and \mbox{$J(2) \to [a_{\B+1}, q_{\B+1}^2]$}.
		Therefore $A = f^{\B+2}_{\B+1} \big ( F \cap J(0) \big )$ and $B = f^{\B+2}_{\B+1} \big ( F \cap J(1) \big )$ are dense in $[a_{\B+1}, q_{\B+1}^1]$
		and $f^{\B+2}_{\B+1}  \big (F \cap J(2) \big )$ is dense in $[a_{\B+1}, q_{\B+1}^2]$.
		Since $[a_{\B+1}, q_{\B+1}^1] \subset [a_{\B+1}, q_{\B+1}^2]$ has dense interior 
		we see $C = [a_{\B+1}, q_{\B+1}^1] \cap f^{\B+2}_{\B+1}  \big ( F \cap J(2) \big )$ is also dense in $[a_{\B+1}, q_{\B+1}^1]$. 
		
		By assumption $f^{\B+2}_{\B+1}$ induces a bijection $ D^2_{\B+2} \to D^2_{\B +1}$ and therefore $A$, $B$ and $C$ are pairwise disjoint.
		Since $F$ is contained in $J(0) \cup J(1) \cup J(2)$ we see $A \cup B \cup C = [a_{\B+1}, q_{\B+1}^1] \cap f^{\B+2}_{\B+1}(F)$.
		We conclude $[a_{\B+1}, q_{\B+1}^1] \cap f^{\B+2}_{\B+1}(F)$ is a resolvable subset of $[a_{\B+1}, q_{\B+1}^1]$. 
		
		By definition $F$ differs from $D^2_{\B+2}$ by only finitely many elements.
		Thus $f^{\B+2}_{\B+1}(F)$ differs from $f^{\B+2}_{\B+1}\big (D^2_{\B+2} \big ) = D^2_{\B+1}$ by only finitely many elements.
		The same holds taking intersections with $[a_{\B+1}, q_{\B+1}^1]$.
		The result then follows from Lemma \ref{irresolvabledifference}.
	\end{proof}
	
	Claim \ref{resolvableStep} says there is no way in general to define a suitable tailing family on $X(\B+2)$.
	For suppose $X(\B+1)$ and $[a_{\B+1}, q_{\B+1}^1]$ are non-metric.
	There is no guarantee a given dense subset of $[a_{\B+1}, q_{\B+1}^1]$ is resolvable.
	Thus a positive answer to Question 3 would go beyond the methods here.
	
	We close with the remark that irresolvable sets can be generated by trying to extend
	the system $\{X(\A);f^\A_\B: \A,\B< \W_1\}$ of continua from Sections \ref{5.3} and \ref{5.4} beyond $\W_1$.
	
	Suppose $X = X(\W_1)$ and $\cD_{\W_1}$ are constructed and $X(\W_1+1)$ is as defined in Section \ref{5.3}.
	For brevity write $a = a_{\W_1}$ and $q^n = q^n_{\W_1}$ and $\cD = \cD_{\W_1}$.
	Suppose all countable dense subsets of each $[a, q^n]$ are resolvable.
	We give only the construction of $D_{\W_1+1}^2 \subset [a_{\W_1+1}, q^2_{\W_1+1}] = J(0) \cup J(1) \cup J(2)$
	as the construction of $D_{\W_1+1}^n$ is analogous.
	
	First we split $[a, q^1] \cap D^2 = A \cup B \cup C$ into three pairwise disjoint dense subsets.
	Recall $J(0)$ and $J(1)$ are copies of $[a, q^1]$ and $J(2)$ a copy of $[a, q^2]$.
	Let $A' \subset J(0)$ and $B' \subset J(1)$ be the dense sets corresponding to $A$ and $B$ and
	\mbox{$C' \subset J(2)$} the dense set corresponding to $D^2 - (A \cup B)$.
	Define $D_{\W_1+1}^2 = A' \cup B' \cup C'$.

	
	In case all relevant countable dense subsets of $[a_{\W_1+1}, q^n_{\W_1+1}]$ are resolvable 
	we define $X(\W_1+2)$ and $\cD_{\W_1+2}$ similarly.
	Proceeding under the same assumption we continue to extend the system.
	Theorem 2 says we need no further assumptions to extend to limit ordinals.
	We claim we cannot extend to $\{X(\A); f^\A_\B: \A,\B< \WW\}$ for any $|\WW| > 2^\cn$.
	
	For then Theorem 2 says the limit $X(\WW)$ is separable.
	But at the same time Corollary 1.12 of \cite{CSbook} says every separable compactum is the image of $\B \NN$ hence has at most $|\B \NN| = 2^\cn$ points.
	But since the chain of subcontinua $X(\A) \subset X(\WW)$ is strictly increasing 
	$X(\WW)$ has cardinality at least $|\WW| > 2^\cn$ which is a contradiction.
	
	We conclude the process terminates at $\{X(\A); f^\A_\B: \A,\B< \eta\}$ for some ordinal $\eta$ with $|\eta| \le 2^\cn$.
	Thus some relevant countable dense subset $D$ of some $[a_{\eta}, q^n_{\eta}]$ is irresolvable.
	
	Closer inspection of the proposed construction of each $D^n_{\eta +1}$ reveals some such $D$ is built out of $D^N_\eta$ for some $n \le N$ 
	by successively splitting sets into two dense disjoint halves or taking intersections with $[a_{\eta}, q^m_{\eta}]$ for some $n \le m \le N$.
	
	\section{Smith's Limit is Not Separable}
	
	\noindent 
	The fundamental difference between our construction and that of Smith \cite{Smith1} is the use of identifications.
	We obtain $X(\B+1)$ from the set
	
	\begin{center}
		$\displaystyle N =  \Big( \bigcup_{n \in \NN}  [p_1,p_n] \times \{1/n\}  \Big ) \cup \big ( Y \times \{0\} \big )$.
	\end{center}
	
	by making for each $n \in \NN$ the identifications $\big (q^n_\B, 1/(2n-1) \big ) \sim \big (q^n_\B, 1/2n \big )$
	and $\big (a_\B, 1/2n \big ) \sim \big (a_\B, 1/(2n+1) \big )$.
	Smith instead obtains $X(\B+1)$ from $N$ by attaching horizontal arcs 
	from $\big (q^n_\B, 1/(2n-1) \big )$ to $\big (q^n_\B, 1/2n \big )$ and from $\big (a_\B, 1/2n \big )$ to $\big (a_\B, 1/(2n+1) \big )$
	as in Figure \ref{SeparablyBellamySuccessorArcs1Input}.

	\begin{figure}[!h]
		\centering
		
		\begin{tikzpicture}

		\usetikzlibrary{snakes}

		\begin{scope}[ xscale = -40]


		\draw[ultra thick, magenta] (1/25,0) -- (1/25,8);
		
		\foreach \n in {2,3,4,5,6}
		{
			\draw (1/2/\n,0) -- (1/2/\n,8-8/\n);
			\draw ({1/(2*\n-1)},0) -- ({1/(2*\n-1)},8-8/\n);
			
		}

		\foreach \n in {2,3,4,5,6}
		{
			\draw[ red] (1/2/\n,8-8/\n) -- ({1/(2*\n-1)},8-8/\n);
			\draw[red] ({1/(2*\n)},0) -- ({1/(2*\n+1)},0);
		}
		
		\draw[dotted] (1/20 ,4) -- (1/13,4);

		\node[circle, fill = orange, minimum width = 0.15cm, inner sep = 0cm] at (1/3,0) {};
		\node[below] at (1/3,0-0.25) {$a_{\B+1} = (a_\B,1)$};
		\node[circle, fill = orange, minimum width = 0.15cm, inner sep = 0cm] at (1/3,4) {};
		\node[above] at (1/3,4+0.25) {$(q^1_\B,1)$};
		
		\node[circle, fill = orange, minimum width = 0.15cm, inner sep = 0cm] at (1/4,0) {};
		\node[below] at (1/4,0-0.25) {$(a_\B,1/2)$};
		\node[circle, fill = orange, minimum width = 0.15cm, inner sep = 0cm] at (1/4,4) {};
		\node[above] at (1/4,4+0.25) {$(q^1_\B,1/2)$};
		
		\node[circle, fill = orange, minimum width = 0.15cm, inner sep = 0cm] at (1/5,0) {};
		\node[below] at (1/5,0-0.25) {$(a_\B,1/3)$};
		\node[circle, fill = orange, minimum width = 0.15cm, inner sep = 0cm] at (1/5,8-8/3) {};
		\node[above] at (1/5,8-8/3+0.25) {$(q^2_\B,1/3)$};
		\node[circle, fill = orange, minimum width = 0.15cm, inner sep = 0cm] at (1/6,8-8/3) {};

		\node[circle, fill = orange, minimum width = 0.15cm, inner sep = 0cm] at (1/25,4) {};
		\node[right] at (1/30,4+0.25) {$q^1_\B$};
		
		\node[circle, fill = orange, minimum width = 0.15cm, inner sep = 0cm] at (1/25,8-8/3) {};
		\node[right] at (1/30,8-8/3+0.25) {$q^2_\B$};

		\node[circle, fill = orange, minimum width = 0.15cm, inner sep = 0cm] at (1/25,0) {};
		\node[right] at (1/30,0) {$a_\B$};

		\node[left] at (1/3 + 0.005,2) {$J(0)$};
		
		\node[left]  at (1/4 + 0.005,2) {$J(1)$};
		
		\node[left]  at (1/5 + 0.0025,4-4/3) {$J(2)$};
		
		\node[left]  at (1/6 ,4-4/3) {$J(3)$};
		
		\node[above] at (1/25,8.25) {$X(\B)$};

		\end{scope}
		
		\end{tikzpicture}
		
		\vspace{0.5cm}
		
		\caption{Schematic of the new $X(\B+1)$. The arcs $I_{\B+1}(n)$ are shown in red.}\label{SeparablyBellamySuccessorArcs1Input}
		
	\end{figure}
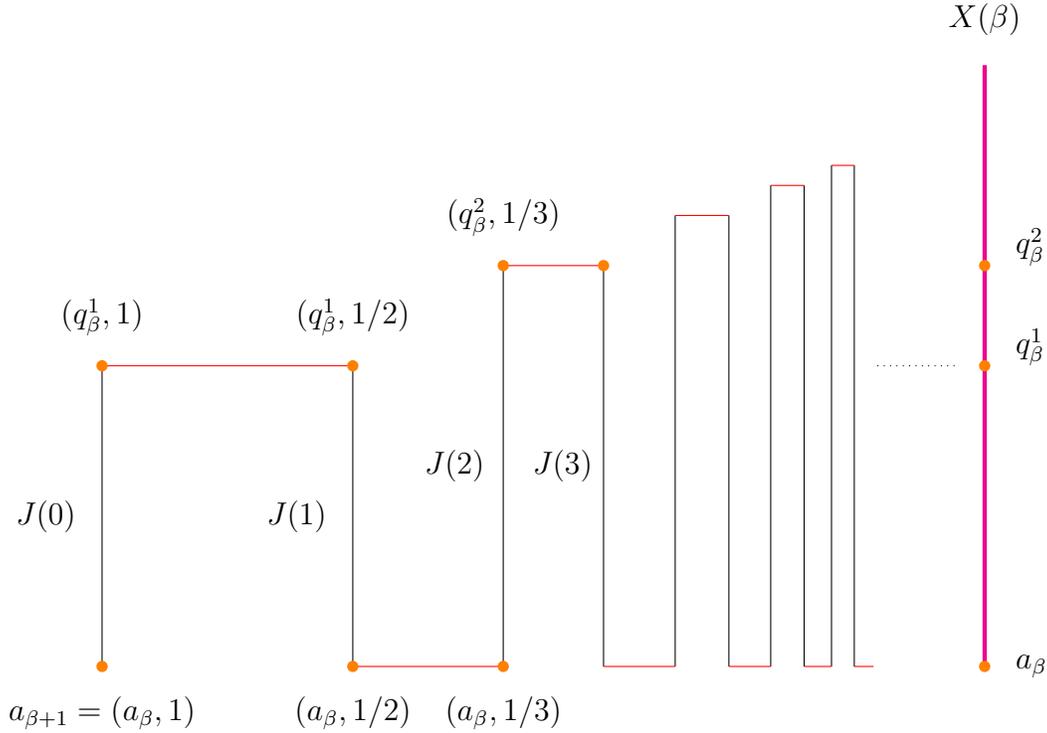

	Suppose as in Section \ref{5.4} we define each limit $X(\A)$ as the inverse limit of all previous $X(\B)$.
	However, unlike before, we define each successor $X(\B +1)$ as the union of the three subspaces of $X(\B) \times [0,1]$.

	\begin{center}
		$ \big ( \bigcup \big \{ [a_\B,q^n_\B] \times \{1/(2n-1),1/2n\} : n \in \NN \big \} \big ) \cup \big ( X(\B) \times \{0\} \big )$
		$\bigcup \big \{ \{q^n_\B\} \times [1/2n, 1/(2n-1)] : n \in \NN \big \}$
		
		$\bigcup \big \{ \{a_\B\} \times [1/(2n+1), 1/2n] : n \in \NN \big \}$
	\end{center}
	
	For each $n \in \NN$ define the arcs $I_{\B+1}({2n-1}) = \{q^n_\B\} \times [1/2n, 1/(2n-1)]$ 
	and $I_{\B+1}(2n) = \{a_\B\} \times [1/(2n+1), 1/2n]$.

	\begin{claim5}\label{notsep}
		The new limit $X$ is $\aleph_1$-cellular hence non-separable.
	\end{claim5}
	
	\begin{proof}
		Let each $V_{\A+1} \subset X(\A)$ be the interior of $I_{\A+1}(2)$.
		Observe $f^{\A+1}_\A(V_{\A+1}) = \{a_\A$\} thus each $f^{\A+1}_{\B}(V_{\A+1}) = \{a_{\B}\}$ for $\B < \A+1$.
		We claim the open sets $U_{\A+1} = \pi_{\A+1}^{-1}(V_{\A+1})$ of $X$ are pairwise disjoint.
		For suppose $x \in U_{\A+1}$. 
		Then by definition $x_{\A+1} \in V_{\A+1}$ and by commutativity $x_{\B+1} \in f^{\A+1}_{\B+1}(V_{\A+1})$ which equals $ \{a_{\B+1}\}$. Thus $x_{\B+1} = a_{\B+1}$. But since $a_{\B+1} \notin V_{\B+1}$ we must have $x \notin U_{\B+1}$
		for each $\B+1 < \A+1$.
		
		Since $\W_1$ is a limit ordinal the map by $\B \mapsto \B+1$ from $\W_1$ to itself is well defined and injective.
		Thus $\{U_{\A+1}: \A < \W_1\}$ has the same cardinality as $\W_1$ namely $\aleph_1$.
		We conclude $X$ is $\aleph_1$-cellular.
	\end{proof}
	
	The proof for the example of Smith \cite{Smith1} is similar, 
	with the interiors of the sets $a^\A_1 \times [1/2,1]$ (page 594) playing the role of $V_{\A+1}$
	and the points $1_\A$ (page 595) playing the role of $a_\A$.

	\section*{Acknowledgements}
	This research was supported by the Irish Research Council Postgraduate Scholarship Scheme grant number GOIPG/2015/2744. 
	The author would like to thank Professor Paul Bankston and Doctor Aisling McCluskey for their help in preparing the manuscript.

\end{document}